\documentclass[11pt,leqno]{amsart}

\usepackage{amsmath}
\usepackage{amsfonts,amssymb,mathtools}
\usepackage[mathcal]{eucal}
\usepackage{hyperref}
\usepackage[nice]{nicefrac}

\theoremstyle{plain}
\newtheorem{theorem}{Theorem}[section]
\newtheorem*{theorem*}{Theorem}
\newtheorem{proposition}[theorem]{Proposition}
\newtheorem{corollary}[theorem]{Corollary}
\newtheorem{lemma}[theorem]{Lemma}

\theoremstyle{remark}
\newtheorem{remark}[theorem]{Remark}

\numberwithin{equation}{section}

\DeclareMathOperator{\hdim}{hdim}
\DeclareMathOperator{\hspec}{hspec}

\newcommand{\F}{\mathbb{F}}
\newcommand{\Z}{\mathbb{Z}}
\newcommand{\N}{\mathbb{N}}

\title[A group with full Hausdorff spectra]{A pro-$p$ group with full
  normal Hausdorff spectra}

\author[I. de las Heras]{Iker de las Heras} 
\address{Iker de las Heras: Department of Mathematics, University
  of the Basque Country UPV/EHU, 48080 Bilbao, Spain}
\email{iker.delasheras@ehu.eus}

\author[B. Klopsch]{Benjamin Klopsch} 
\address{Benjamin Klopsch:
  Mathematisches Institut, Heinrich-Heine-Universit\"at, 40225
  D\"usseldorf, Germany} \email{klopsch@math.uni-duesseldorf.de}

\thanks{The first author acknowledges support by the Spanish
  Government, grant MTM2017-86802-P, partly with FEDER funds, and by
  the Basque Government, grant IT974-16.  Furthermore he acknowledges
  a predoctoral grant of the University of the Basque Country.}



\begin{document}

\begin{abstract}
  For each odd prime~$p$, we produce a $2$-generated pro-$p$ group $G$
  whose normal Hausdorff spectra
  \[
  \hspec_{\trianglelefteq}^{\mathcal{S}}(G) =
  \{ \hdim_{G}^{\mathcal{S}}(H)\mid H\trianglelefteq_\mathrm{c} G \}
  \]
  with respect to five standard filtration series $\mathcal{S}$ --
  namely the lower $p$-series, the dimension subgroup series, the
  $p$-power series, the iterated $p$-power series and the Frattini
  series -- are all equal to the full unit interval~$[0,1]$.   Here
  $\hdim_G^{\mathcal{S}} \colon \{ X\mid X \subseteq G \} \to[0,1]$
  denotes the Hausdorff dimension function associated to the natural
  translation-invariant metric induced by the filtration
  series~$\mathcal{S}$.
\end{abstract}

\maketitle


\section{Introduction}

The concept of Hausdorff dimension has led to interesting results in
the theory of profinite groups; for instance, see~\cite{KlThZR19} and
the references therein.  Let $G$ be an infinite countably based
profinite group and let $\mathcal{S}$ be a \emph{filtration series} of
$G$, that is, a chain $G = S_0\ge S_1 \ge S_2 \ge \ldots$ of open
normal subgroups $S_i \trianglelefteq_\mathrm{o} G$ such that
$\bigcap_i S_i = 1$.  These subgroups form a base of neighbourhoods of
$1$ and induce a translation-invariant metric on $G$ which, in turn,
associates a \emph{Hausdorff dimension}
$\hdim_G^{\mathcal{S}}(U) \in [0,1]$ to any subset $U\subseteq G$ with
respect to the filtration series $\mathcal{S}$.

Barnea and Shalev~\cite{BaSh97} established a group-theoretical
interpretation of $\hdim_G^{\mathcal{S}}(H)$ for closed subgroups
$H \le_\mathrm{c} G$; they showed that
\[
\hdim_G^{\mathcal{S}}(H)=\varliminf_{i\rightarrow\infty}\dfrac{\log_p
  \lvert HS_i : S_i \rvert}{\log_p \lvert G : S_i \rvert}
\]
can be regarded as a `logarithmic density' of $H$ in~$G$.  The
(ordinary) \emph{Hausdorff spectrum} of~$G$ is
$\hspec^{\mathcal{S}}(G) = \{\hdim_{G}^{\mathcal{S}}(H)\mid H
\le_\mathrm{c} G\}$.
The \emph{normal Hausdorff spectrum} of~$G$, defined as
\[
\hspec_{\trianglelefteq}^{\mathcal{S}}(G)=\{\hdim_{G}^{\mathcal{S}}(H)\mid H\trianglelefteq_\mathrm{c} G\},
\]
provides a snapshot of the normal subgroup structure of~$G$; its
significance was highlighted by Shalev in~\cite[\S 4.7]{Sh00}.  

Typically, the Hausdorff dimension function and the normal Hausdorff
spectrum depend very much on the underlying filtration~$\mathcal{S}$;
compare~\cite{KlThZR19}.  For a finitely generated pro-$p$ group~$G$,
there are natural choices for $\mathcal{S}$ that encapsulate
group-theoretic properties of~$G$: the lower $p$-series $\mathcal{L}$,
the dimension subgroup series~$\mathcal{D}$, the $p$-power
series~$\mathcal{P}$, the iterated $p$-power series~$\mathcal{P}^*$,
and the Frattini series~$\mathcal{F}$; see Section~\ref{sec:prelim}.
We refer to these filtration series loosely as the five standard
filtration series.

Several types of profinite groups with full ordinary Hausdorff spectra
$[0,1]$ have been identified.  The first examples of finitely
generated pro-$p$ groups $G$ with $\hspec^{\mathcal{P}}(G) = [0,1]$
were discovered by Levai (see \cite[\S 4.2]{Sh00}) and Klopsch~\cite[VIII, \S
7]{Kl99}; more complicated examples of profinite groups with full
Hausdorff spectra can be found, for example, in~\cite{AbVi05, BaVa19,
  GaGaKlxx}.  But until now no examples of finitely generated pro-$p$
groups with full normal Hausdorff spectra were known.

Already twenty years ago, Shalev~\cite[\S 4.7]{Sh00} put up the
challenge to construct finitely generated pro-$p$ groups with infinite
normal Hausdorff spectra and he asked whether the normal Hausdorff
spectra could even contain infinite real intervals.  Recently, Klopsch
and Thillaisundaram~\cite{KlTh19} succeeded in constructing such
examples, with respect to the five standard filtration series.  Even
though the normal Hausdorff spectra of their groups each contain
infinite intervals, none of the spectra covers the full unit
interval~$[0,1]$.  In this paper we modify the construction of Klopsch
and Thillaisundaram to produce the first example of a finitely
generated pro-$p$ group with full normal Hausdorff spectrum $[0,1]$,
with respect to any of the five standard filtration series.

Our construction proceeds as follows.  Throughout, let $p$ denote an
odd prime.  For $k\in \N$, consider the finite wreath product 
\begin{multline*}
  W_k = B_k \rtimes \langle \dot x_k \rangle \cong \langle \dot y_k\rangle \wr
  \langle \dot x_k \rangle, \qquad \text{with cyclic top group $\langle
    \dot x_k\rangle\cong C_{p^k}$}\\
  \text{and elementary abelian base group
    $B_k = \prod_{j = 0}^{p^k-1} \langle {\dot y_k}^{\,\dot  x_k^{\, j}}\rangle
    \cong C_p^{\, p^k}$.}
\end{multline*}
Basic structural properties of the finite wreath products $W_k$
transfer naturally to the inverse limit $W \cong \varprojlim_k W_k$,
i.e., the pro-$p$ wreath product
\begin{multline*}
  W = \langle \dot x, \dot y \rangle = B \rtimes \langle \dot x \rangle\cong C_p\
  \hat{\wr}\ \Z_p \qquad \text{with procyclic top group
    $\langle \dot x\rangle \cong
    \Z_p$} \\
  \text{and elementary abelian base group
    $B = \overline{\langle {\dot y}^{\dot x^j} \mid j \in\Z \rangle} \cong C_p^{\, \aleph_0}$.}
\end{multline*}
 
Let $F = F_2 = \langle \tilde x, \tilde y \rangle$ be a free pro-$p$
group on two generators, and let $\eta \colon F \to W$, resp.\
$\eta_k \colon F \to W_k$, for $k \in \N$, denote the continuous
epimorphisms induced by $\tilde x \mapsto \dot x$ and
$\tilde y \mapsto \dot y$, resp.\ $\tilde x \mapsto \dot x_k $ and
$\tilde y \mapsto \dot y_k$.  Set
$R = \mathrm{ker}(\eta) \trianglelefteq_\mathrm{c} F$ and
$R_k = \mathrm{ker}(\eta_k) \trianglelefteq_\mathrm{o} F$; set
$Y = B \eta^{-1} \trianglelefteq_\mathrm{c} F$ and
$Y_k = B_k \eta_k^{\, -1} \trianglelefteq_\mathrm{o} F$.  We define
\begin{align*}
  G  = F/N, & \quad \text{where $N = [R,Y]Y^p \trianglelefteq_\mathrm{c}
              F$}, \\
  G_k  = F/N_k, & \quad  
                  \text{where $N_k = [R_k,Y_k]Y_k^{\, p} \langle
                  {\tilde x}^{p^k} \rangle^F$.}
\end{align*}
Furthermore, we write 
\begin{align*}
  H = Y/N \trianglelefteq_\mathrm{c} G & \qquad \text{and} \qquad Z = R/N
                                         \trianglelefteq_\mathrm{c} G, \\
  H_k = Y_k/N_k \trianglelefteq G_k & \qquad \text{and} \qquad
                                      Z_k = R_k / N_k \trianglelefteq G_k.
\end{align*}
We denote the images of $\tilde x, \tilde y$ in~$G$, resp.\ in~$G_k$,
by $x,y$, resp.\ $x_k,y_k$, so that
$G = \overline{\langle x,y\rangle}$ and $G_k=\langle x_k,y_k\rangle$.

We observe that the finite groups $G_k$, $k\in\N$, naturally form an
inverse system and that $G \cong \varprojlim_k G_k$.  Furthermore, we have
$[H,Z]=1$, and $[H_k,Z_k]=1$ for all $k \in \N$.  

\begin{theorem}\label{thm:main}
  For $p > 2$, the $2$-generated pro-$p$ group $G$ constructed above
  has full normal Hausdorff spectra with respect to the five standard
  filtration series:
  \[
  \hspec_{\trianglelefteq}^{\mathcal{L}}(G) =
  \hspec_{\trianglelefteq}^{\mathcal{D}}(G) =
  \hspec_{\trianglelefteq}^{\mathcal{P}}(G) =
  \hspec_{\trianglelefteq}^{\mathcal{P}^*}(G) =
  \hspec_{\trianglelefteq}^{\mathcal{F}}(G) = [0,1].
  \]
\end{theorem}

This resolves Problems~1.2 (b),(c) in~\cite{KlTh19} and Problem~5
in~\cite{BaSh97} for all five standard series.  The latter problem was
already solved previously for the series $\mathcal{D}$, $\mathcal{P}$,
$\mathcal{P}^*$ and~$\mathcal{F}$: in \cite[VIII, \S 7]{Kl99} it was
seen that $W \cong C_p\ \hat{\wr}\ \Z_p$ has
$\hspec^{\mathcal{D}}(W) = \hspec^{\mathcal{P}}(W) =
\hspec^{\mathcal{F}}(W) = [0,1]$,
and by completely different means it was shown in~\cite{GaGaKlxx} that
a non-abelian finitely generated free pro-$p$ group $E$ has
$\hspec^{\mathcal{D}}(E) = \hspec^{\mathcal{P}^*}(E) =
\hspec^{\mathcal{F}}(E) = [0,1]$.

\smallskip

\noindent
\textit{Notation.} Throughout, $p$ denotes an \emph{odd} prime.  From
now on, all subgroups of profinite groups are tacitly understood to be
closed subgroups to simplify the notation.  All iterated commutators
are left-normed, e.g., $[x,y,z] = [[x,y],z]$.

Section~$2$ contains basic material and fairly general considerations
that do not yet involve the notation used in the construction of the
particular groups $G$ and $G_k$, $k \in \N$.  

In Sections~\ref{sec:structure-G-k} and \ref{sec:normal-hspec} we use
the special notation from the introduction.  In addition, we write
$c_1=y$ and $c_i=[y,x,\overset{i-1}{\ldots},x]$ for
$i \in \N_{\ge 2}$; furthermore, we set $c_{i,1}=[c_i,y]$ and
$c_{i,j}=[c_i,y,x,\overset{j-1}{\ldots},x]$ for $j \in \N_{\ge 2}$.
To keep the notation manageable, we denote, for $k \in \N$, the
corresponding elements in the finite group $G_k$ by the same symbols
(suppressing the parameter~$k$): $c_1=y_k$ and
$c_i=[y_k,x_k,\overset{i-1}{\ldots},x_k]$ for $i\in\N_{\ge 2}$, and
similarly $c_{i,1}=[c_i,y_k]$ and
$c_{i,j}=[c_i,y_k,x_k,\overset{j-1}{\ldots},x_k]$ for
$j \in \N_{\ge 2}$.  From the context it will be clear whether our
considerations apply to $G$ or one of the groups~$G_k$.


\section{Preliminaries} \label{sec:prelim}

Let $G$ be an arbitrary finitely generated pro-$p$ group.  We recall
the definition of the five standard filtration series referred to in
the introduction.  The \emph{lower $p$-series} $\mathcal{L}$ of~$G$,
the \emph{dimension subgroup series} $\mathcal{D}$ of~$G$, the
\emph{$p$-power series} $\mathcal{P}$ of~$G$, the \emph{iterated
  $p$-power series} $\mathcal{P}^*$ of~$G$ and the \emph{Frattini
  series} $\mathcal{F}$ of~$G$ are defined recursively by
\begin{align*}
  \mathcal{L} & \colon P_1(G)=G\ \ \text{ and }\ \
                P_i(G)=P_{i-1}(G)^p[P_{i-1}(G),G]\ \ \text{for $i\ge
                2$,} \\
  \mathcal{D} & \colon D_1(G)=G\ \ \text{ and }\ \ D_i(G) = D_{\lceil
                i/p\rceil}(G)^p \prod\nolimits_{1\le j<i}[D_{j}(G),D_{i-j}(G)]\ \
                \text{for $i\ge 2$}, \\
  \mathcal{P} & \colon  \pi_i(G) = G^{p^i}=\langle g^{p^i}\mid g\in
                G\rangle \ \text{for $i \ge 0$}, \\
  \mathcal{P}^* & \colon  \pi^*_0(G)=G\ \ \text{ and }\ \
                  \pi^*_i(G) = \pi^*_{i-1}(G)^p\ \text{for $i \ge 1$,} \\
  \mathcal{F} & \colon \Phi_0(G)=G\ \ \text{ and }\ \
                \Phi_i(G)=\Phi_{i-1}(G)^p[\Phi_{i-1}(G),\Phi_{i-1}(G)]\
                \text{for $i \ge 1$.}
\end{align*}

Next we recall two standard commutator identities;
compare~\cite[Prop.~1.1.32]{LeMK02}.

\begin{lemma} \label{lem:com-ids} Let $G=\langle a, b\rangle$ be a
  finite $p$-group, for $p \ge 3$, such that $\gamma_2(G)$ has
  exponent~$p$, and let $r\in\N$.  For $u,v\in G$, let $K(u,v)$ denote
  the normal closure in $G$ of all commutators in $\{u,v\}$ of weight
  at least $p^r$ that have weight at least $2$ in~$v$.  

  Then the following congruences hold:
  \[
    (ab)^{p^r} \equiv_{K(a,b)} a^{p^r} b^{p^r}
    [b,a,\overset{p^r-1}{\ldots},a] \qquad \text{and} \qquad
    [a^{p^r},b] \equiv_{K(a,[a,b])} [a,b,a,\overset{p^r-1}{\ldots},a].
  \]
\end{lemma}

The main ingredient of the proof of Theorem~\ref{thm:main} is
Proposition~\ref{pro:path-area}.  For the proof we first establish two
lemmata.  The first lemma is a variation
of~\cite[Prop.~5.2]{KlThZR19}.

\begin{lemma} Let $G$ be a countably based pro-$p$ group, and let
  $Z \trianglelefteq_\mathrm{c} G$ be infinite. Let
  $\mathcal{S} \colon Z_0 \supseteq Z_1 \supseteq \ldots$ be a
  filtration series of~$Z$ consisting of $G$-invariant subgroups
  $Z_i \trianglelefteq_\mathrm{o} Z$.  Let $\eta \in [0,1]$ be such
  that the normal closure in $G$ of every finite collection of
  elements $z_1, \ldots, z_m \in Z$ satisfies
  $\hdim_Z^\mathcal{S}(\langle z_1, \ldots, z_m \rangle^G) \le \eta$.

  Then there exists $H \le_\mathrm{c} Z$ with $H \trianglelefteq G$
  such that $\hdim_Z^\mathcal{S}(H) = \eta$.
\end{lemma}

\begin{proof} 
  The claim can be verified in close analogy to the proof
  of~\cite[Prop.~5.2]{KlThZR19}.  One constructs the subgroup
  $H \le_\mathrm{c} Z$ as
  $H = \langle H_0 \cup H_1 \cup \ldots \rangle$, where
  $1 = H_0 \subseteq H_1 \subseteq \ldots$ is a suitable ascending
  sequence of subgroups $H_i \le_\mathrm{c} Z$ each of which is the
  normal closure in $G$ of finitely many elements.  To see that the
  argument in op.\ cit.\ can be used, it suffices to observe that, for
  each $i \in \N$, the pro-$p$ group $G/Z_i$ acts nilpotently on the
  finite $p$-group~$Z/Z_i$ (and its quotients by $G$-invariant
  subgroups).
\end{proof}

\begin{lemma} \label{lemma path/area} Let $G$ be a countably based
  profinite group with an infinite abelian normal subgroup
  $Z \trianglelefteq_\mathrm{c} G$ and $x \in G$ such that
  $G = \langle x \rangle C_G(Z)$.  Let
  $\mathcal{S}:Z=Z_0\ge Z_1\ge\ldots$ be a filtration series of $Z$
  consisting of $G$-invariant subgroups
  $Z_i \trianglelefteq_\mathrm{o} Z$; for $i \in \N_0$, let $p^{e_i}$
  be the exponent of~$Z/Z_i$.  Suppose that, for every $i \in \N_0$,
  there exist $n_i \in \N$ and $N_i \le_\mathrm{c} Z$ such that
  \[
  \gamma_{n_i +1}(G)\cap Z\le Z_i\le N_i \qquad \text{and} \qquad
  \varliminf_{i\rightarrow\infty} \frac{e_in_i}{\log_p \lvert Z : N_i
    \rvert}=0.
  \]
  
  Then every finite collection of elements $z_1, \ldots, z_m \in Z$
  satisfies
  \[
  \hdim_Z^\mathcal{S}(\langle z_1, \ldots, z_m \rangle^G) = 0.
  \]
\end{lemma}

\begin{proof}
  Consider first a single element $z \in Z$.  From
  \[
  \langle z\rangle^G=\langle z,[z,x],[z,x,x],\ldots\rangle,
  \]
  and $\gamma_{n_i +1}(G)\cap Z\le Z_i$, for $i \in \N$, we deduce that
 \[
 \langle z\rangle^G Z_i = \langle
 z,[z,x],\ldots,[z,x,\overset{n_i-1}{\ldots},x]\rangle Z_i;
  \]
  in particular, since $Z$ is abelian, this yields
  \[
  \log_p \lvert \langle z \rangle^G Z_i : Z_i \rvert \le e_in_i.
  \]

  Now consider finitely many elements $z_1, \ldots, z_m \in Z$.  Since
  $Z$ is abelian, we have
  $\langle z_1,\ldots,z_m \rangle^G = \langle z_1\rangle^G \cdots
  \langle z_m\rangle^G$.  From this we deduce
   \[
   \hdim_Z^{\mathcal{S}}(\langle z_1,\ldots,z_m \rangle^G) \leq
     \varliminf_{i\rightarrow\infty}\frac{ \sum_{j=1}^m \log_p \lvert \langle z_j
    \rangle^G Z_i : Z_i \rvert}{\log_p \lvert Z : Z_i \rvert} \le
   \varliminf_{i\rightarrow\infty} \frac{m e_i n_i}{\log_p \lvert Z : N_i
   \rvert} = 0. \qedhere
 \]
\end{proof}

For an infinite countably based pro-$p$ group $G$, equipped with a
filtration series
$\mathcal{S} \colon G = S_0 \supseteq S_1\supseteq \ldots$, and a
closed subgroup $H \le_\mathrm{c} G$ we adopt the following
terminology from~\cite{KlTh19}: we say that $H$ has \emph{strong}
Hausdorff dimension in $G$ with respect to $\mathcal{S}$ if its
Hausdorff dimension is given by a proper limit, i.e., if
\[
\hdim^{\mathcal{S}}_G(H) = \lim_{i \to \infty} \frac{\log_p \lvert H
  S_i : S_i \rvert}{\log_p \lvert G : S_i \rvert}.
\]

Using the previous two lemmata, we follow the proof
of~\cite[Thm.~5.4]{KlThZR19} to obtain our main tool.

\begin{proposition} \label{pro:path-area} Let $G$ be a countably based
  pro-$p$ group with an infinite abelian normal subgroup
  $Z \trianglelefteq_\mathrm{c} G$ such that $G/C_G(Z)$ is procyclic.
  Let $\mathcal{S} \colon G = S_0\ge S_1\ge \ldots$ be a filtration
  series of $G$ and consider the induced filtration series
  $\mathcal{S} |_Z \colon Z = S_0 \cap Z \ge S_1 \cap Z \ge \ldots$ of
  $Z$; for $i\in\N_0$, let $p^{e_i}$ be the exponent of
  $Z/(S_i \cap Z)$.  Suppose that, for every $i\in \N_0$, there exist
  $n_i \in \N$ and $M_i \le_\mathrm{c} G$ such that
  \[
  \gamma_{n_i + 1}(G) \cap Z \le S_i \cap Z \le M_i \qquad
  \text{and} \qquad
  \varliminf_{i\rightarrow\infty}\frac{e_in_i}{\log_p \lvert Z : M_i
    \cap Z \rvert} = 0.
  \]

  If $Z$ has strong Hausdorff dimension
  $\xi = \hdim_G^{\mathcal{S}}(Z) \in [0,1]$ then we have
  \[ 
  [0,\xi]\subseteq \hspec_{\trianglelefteq}^{\mathcal{S}}(G).
  \]
\end{proposition}


\section{The structure of the finite groups
  $G_k$} \label{sec:structure-G-k}

In this section we collect some structural results for the finite
$p$-groups $G_k$ defined in the introduction.  We use the notation set
up there, in particular, in the last paragraph of that section: $W_k$,
$B_k$, $\dot x_k$, $\dot y_k$, $G_k$, $H_k$, $Z_k$, $x_k$, $y_k$,
$c_i$, $c_{i,j}$, \ldots.

\begin{proposition}[Prop.~2.6 in \cite{KlTh19}] \label{pro:Wk-lcs} For
  $k\in\N$, the wreath product $W_k \cong C_p \wr C_{p^k}$ is
  nilpotent of class $p^k$ and the lower central series of $W_k$
  satisfies
  \begin{align*}
    W_k & =\gamma_1(W_k) = \langle \dot x_k, \dot y_k \rangle
          \gamma_2(W_k) \text{ with }W_k / \gamma_2(W_k) \cong C_{p^k}
          \times C_{p}, \\
   \gamma_i(W_k) & = \langle [\dot y_k, \dot x_k,
                   \overset{i-1}{\ldots}, \dot
                   x_k]\rangle\gamma_{i+1}(W_k) \text{ with
                   }\gamma_i(W_k)/\gamma_{i+1}(W_k)\cong C_p \text{
                   for } 2\le i\le p^k.
  \end{align*}
  In particular, the base group satisfies
  \[
  B_k = \langle \dot y_k \rangle \gamma_2(W_k) = \langle
  \dot y_k, [\dot y_k, \dot x_k], \ldots, [\dot y_k, \dot x_k,
  \overset{p^k-1}{\ldots}, \dot x_k]\rangle.
  \]
\end{proposition}

\begin{proposition} \label{pro:order-G-k} For $k\in\N$, we have
  $G_k = \langle x_k \rangle \ltimes H_k$, where
  $\langle x_k \rangle \cong C_{p^k}$ and $H_k$ is freely generated in
  the variety of class-$2$ nilpotent groups of exponent~$p$ by the
  conjugates $y_k^{\, x_k^{\, j}}$, $0 \le j < p^k$.  In particular,
  the logarithmic order of $G_k$ is
  \[
  \log_p \lvert G_k \rvert = k+p^k+\binom{p^k}{2}.
  \]
\end{proposition}

\begin{proof}
  The proof is very similar to that of \cite[Lem.~5.1]{KlTh19}.  From
  $G_k/Z_k\cong W_k $ we obtain
  \[
  \log_p \lvert G_k \rvert = \log_p \lvert G_k/Z_k \rvert + \log_p
  \lvert Z_k \rvert = k + p^k + \log_p \lvert Z_k \rvert.
  \]
  By construction, $Z_k$ is elementary abelian, and
  from~\cite[Eq.~(3.1)]{KlTh19} we get
  \[
  Z_k = \Big\langle \big[ y_k^{\, x_k^{p^i}},y_k^{\, x_k^{p^j}} \big]
  \mid 0 \le i < j \le p^{k}-1 \Big\rangle.
  \]
  This yields $\log_p \lvert G_k \rvert \le k+p^k+\binom{p^k}{2}$.  

  Consider the finite $p$-group
  \[
  M = \langle b_0, \ldots, b_{p^k-1} \rangle = E / \gamma_3(E)E^p,
  \]
  where $E$ is the free group on $p^k$ generators.  Then, the images
  of $b_0, \ldots, b_{p^k-1}$ generate independently the elementary
  abelian quotient $M/M'$, and the commutators $[b_i,b_j]$ with
  $0\le i<j\le p^k-1$ generate independently the elementary abelian
  subgroup~$M'$.  The latter can be checked, for instance, by
  considering homomorphisms from $M$ onto the group
  $\mathrm{Heis}(\F_p)$ of upper unitriangular $3\times 3$ matrices
  over the prime field~$\F_p$.  Next consider the faithful action of
  the cyclic group $A \cong \langle a \rangle \cong C_{p^k}$ on~$M$
  induced by
  \[
  b_i^{\, a} =
  \begin{cases}
    b_{i+1} & \text{if $0\le i\le p^k-2$,} \\
    b_0 & \text{if $i=p^{k}-1$.}
  \end{cases}
  \]
  We define $\widetilde{G}_k = A \ltimes M$ and note that
  $\log_p \lvert G_k \rvert \le k + p^k + \binom{p^k}{2} = \log_p
  \lvert \widetilde{G}_k \rvert$.
  Furthermore, it is easy to see that $\widetilde{G}_k/M' \cong W_k$.
  Thus there is an epimorphism
  $\varepsilon \colon G_k \to \widetilde{G}_k$ with
  $x_k \,\varepsilon = a$ and $y_k \,\varepsilon = b_0$, and from
  $\lvert G_k \rvert \le \lvert \widetilde{G}_k \rvert$ we conclude
  that $G_k\cong \widetilde{G}_k$.
\end{proof}

\begin{remark} \label{rem:H'-eq-Z}
  \label{rem:derived-subgroup}
  The proof of Proposition~\ref{pro:order-G-k} shows that
  $[H_k,H_k] = Z_k$ for $k \in \N$, and thus $[H,H] = Z$.
\end{remark}

\begin{proposition} \label{pro:G-k-lcs} For $k\in\N$, the nilpotency
  class of $G_k$ is $2p^k-1$.  The terms of the lower central series
  of $G_k$ are as follows:
  \[
  \gamma_1(G_k)=G_k=\langle x_k,y_k\rangle \, \gamma_2(G_k) \quad
  \text{with $G_k/\gamma_2(G_k)\cong C_{p^k}\times C_p$}
  \]
  and, with the notation
  \begin{align*}
    I_1 & = \{ i \mid 2 \le i \le p^k \text{ with } i
    \equiv_2 0 \}, && I_2  = \{ i \mid 2 \le i \le p^k \text{ with } i
                      \equiv_2 1 \}, \\
    I_3 &= \{ i \mid p^k + 1 \le i \le 2 p^k-1 \text{ with } i
    \equiv_2 0 \}, && I_4  = \{ i \mid p^k + 1 \le i \le 2p^k-1 \text{ with } i
                      \equiv_2 1 \},
  \end{align*}
  the series continues as
  \begin{align*}
    \gamma_i(G_k)  %
    & =
      \begin{cases}
        \langle c_i, \, c_{2,i-2}, \, c_{4,i-4}, \, \ldots, \,
        c_{i-2,2} \rangle \gamma_{i+1}(G_k) & \text{for $i \in I_1$,} \\
        \langle c_i, \, c_{2,i-2}, \, c_{4,i-4}, \, \ldots, \,
        c_{i-1,1} \rangle \gamma_{i+1}(G_k) & \text{for $i \in I_2$,}
        \\
        \langle c_{i-p^k+1,p^k-1}, \, c_{i-p^k+3,p^k-3}, \, \ldots, \,
        c_{p^k-1,i-p^k+1} \rangle \gamma_{i+1}(G_k) & \text{for $i
          \in I_3$,} \\
        \langle c_{i-p^k,p^k}, \, c_{i-p^k+2,p^k-2}, \, \ldots, \,
        c_{p^k-1,i-p^k+1} \rangle \gamma_{i+1}(G_k) & \text{for $i \in I_4$}
      \end{cases}
  \end{align*}
  with
  \begin{align*}
    \gamma_i(G_k)/\gamma_{i+1}(G_k)  %
    & \cong \begin{cases}
      C_p^{\, i/2} & \text{for $i \in I_1$,} \\
      C_p^{\, (i+1)/2} &  \text{for $i \in I_2$,} \\
      C_p^{\,  (2p^k-i)/2} & \text{for $i \in I_3$,} \\
      C_p^{\, (2p^k-i+1)/2}  & \text{for $i \in I_4$.}
    \end{cases}
  \end{align*}
\end{proposition}

\begin{proof}
  The description of $\gamma_1(G_k)$ modulo $\gamma_2(G_k)$ is clear.
  Now consider $i \in I_1$, that is $2 \le i \le p^k$ and
  $i \equiv_2 0$.  Our first aim is to show, by induction on~$i$, that
  \begin{equation} \label{equ:cases-I1-I2}
    \begin{split}
      \gamma_i(G_k) & = \langle c_i, \, c_{2,i-2}, \, c_{4,i-4}, \,
      \ldots, \,
      c_{i-2,2} \rangle \gamma_{i+1}(G_k), \\
      \gamma_{i+1}(G_k) & = \langle c_{i+1}, \, c_{2,i-1}, \,
      c_{4,i-3}, \, \ldots, c_{i,1} \rangle\gamma_{i+2}(G_k).
    \end{split}
  \end{equation}
  The induction base, i.e., the case $i=2$, is clear:
  $\gamma_2(G_k) = \langle [x_k, y_k] \rangle \gamma_3(G_k) = \langle
  c_2 \rangle \gamma_3(G_k)$
  and
  $\gamma_3(G_k) = \langle [c_2,x_k], [c_2,y_k] \rangle \gamma_4(G_k)
  = \langle c_3,c_{2,1}\rangle \gamma_4(G_k)$.
  Next suppose that $i \ge 4$.  The induction hypothesis yields
  \begin{align*}
    \gamma_{i-2}(G_k) & =\langle c_{i-2}, \, c_{2,i-4}, \, c_{4,i-6},
                        \, \ldots, \,
                        c_{i-4,2} \rangle\gamma_{i-1}(G_k), \\
    \gamma_{i-1}(G_k) & = \langle c_{i-1}, \, c_{2,i-3}, \, c_{4,i-5},
                        \, \ldots, \,
                        c_{i-2,1}\rangle\gamma_{i}(G_k).
  \end{align*}
  From $c_{m,n}\in [H_k,H_k] = Z_k$ we deduce $[c_{m,n},y_k]=1$ for
  all $m,n \ge 1$.  This gives
  \[
  \gamma_i(G_k) = \langle
  c_i, \, c_{i-1,1}, \, c_{2,i-2}, \, c_{4,i-4}, \, \ldots, \, c_{i-2,2}\rangle\gamma_{i+1}(G_k).
  \]
  We put
  \[
  M= \langle c_i, \, c_{2,i-2}, \, c_{4,i-4}, \, \ldots,
  c_{i-2,2} \rangle \gamma_{i+1}(G_k)
  \] 
  and aim to show that $c_{i-1,1}\in M$.  This will establish the
  first equation in~\eqref{equ:cases-I1-I2}; the second equation then
  follows immediately, again from $[c_{n,m},y_k]=1$ for $m,n \ge 1$.

  As $c_{i-1,1} = [c_{i-2},x_k,y_k]$, the Hall--Witt identity yields
  \[
  c_{i-1,1} [x_k,y_k,c_{i-2}] [y_k,c_{i-2},x_k] \equiv 1 \pmod{M}.
  \]
  Furthermore, $[y_k,c_{i-2},x_k] \equiv c_{i-2,2}^{\, -1} \equiv 1$
  modulo $M$, and this gives
   \[
   c_{i-1,1}\equiv [c_{i-2},c_2]^{-1}\pmod{M}.
   \]
   Thus it suffices to prove that 
   \[
   [c_m,c_n] \equiv 1 \pmod{M} \qquad \text{for all $m,n \in \N$ with
   $m \ge n\ge 2$ and $m+n = i$.}
   \]
   We argue by induction on $m-n$.  If $m-n=0$ then $m=n$ and
   $[c_m,c_n]=1$.  Now suppose that $m-n\ge 1$.  As
   $[c_m,c_n] = [c_{m-1},x_k,c_n]$, the Hall--Witt identity yields
   \[ 
   [c_m,c_n][x_k,c_n,c_{m-1}][c_n,c_{m-1},x_k]\equiv 1\pmod{M},
   \]
   where
   $[x_k,c_n,c_{m-1}] \equiv [c_{m-1},c_{n+1}] \equiv 1 \pmod{M}$ by
   induction.  This yields
   \[
   [c_m,c_n] \equiv[c_n,c_{m-1},x_k]^{-1}
   \equiv [[c_n,c_{m-1}]^{-1},x_k] \pmod{M}.
   \]
   From $[c_n,c_{m-1}]^{-1} \in \gamma_{i-1}(G_k)$ we deduce that
   \[ 
   [c_n,c_{m-1}]^{-1} \equiv c_{i-1}^{\, r_0} c_{2,i-3}^{\, r_2}
   c_{4,i-5}^{\, r_4} \cdots c_{i-2,1}^{\, r_{i-2}}
   \pmod{\gamma_i(G_k)}
   \]
   for suitable $r_0, r_2, \ldots, r_{i-2} \in \Z$.  It follows that
   \[
   [c_m,c_n] \equiv [[c_n,c_{m-1}]^{-1},x_k] \equiv c_i^{\, r_0}
   c_{2,i-2}^{\, r_2} c_{4,i-4}^{\, r_4} \cdots c_{i-2,2}^{\, r_{i-2}}
   \equiv 1 \pmod{M}.
   \]
   This finishes the proof of~\eqref{equ:cases-I1-I2}.  Finally, we
   observe from~\eqref{equ:cases-I1-I2} that 
   \[
   \gamma_i(G_k)/\gamma_{i+1}(G_k) \cong C_p^{\, l(i)} \qquad
   \text{and} \qquad \gamma_{i+1}(G_k)/\gamma_{i+2}(G_k) \cong C_p^{\,
     l(i+1)},
   \]
   where $l(i) \le i/2$ and $l(i+1) \le i/2 + 1$; below we will see
   that, in fact, all the generators appearing
   in~\eqref{equ:cases-I1-I2} are necessary.

   Now consider $i \in I_3$, that is $p^k +1 \le i \le 2p^k-2$ and
   $i \equiv_2 0$.  Lemma~\ref{lem:com-ids} yields
   \[
   c_{p^k+1} \equiv [y_k,x_k^{\, p^k}] = [y_k,1] = 1
   \pmod{\gamma_{p^k+2}(G_k)},
  \]
  thus $c_{p^k+1}\in\gamma_{p^{k}+2}(G_k)$ and
  $c_{p^k+1,n} \in \gamma_{p^k+n+2}(G_k)$ for $ n\ge 1$.  For similar
  reasons, we have $c_{n,p^k+1}\in \gamma_{p^k+n+2}(G_k)$ for all
  $n\ge 1$.  This yields, by induction on~$i$,
  \begin{equation}\label{equ:cases-I3-I4}
    \begin{split}
      \gamma_i(G_k) & = \langle c_{i-p^k+1,p^k-1}, \,
      c_{i-p^k+3,p^k-3},
      \, \ldots, \, c_{p^k-1,i-p^k+1} \rangle \gamma_{i+1}(G_k), \\
      \gamma_{i+1}(G_k) & = \langle c_{i-p^k+1,p^k}, \,
      c_{i-p^k+3,p^k-2}, \, \ldots, \, c_{p^k-1,i-p^k+2} \rangle
      \gamma_{i+2}(G_k).
    \end{split}
  \end{equation}
  Similarly as before, we observe that
   \[
   \gamma_i(G_k)/\gamma_{i+1}(G_k) \cong C_p^{\, l(i)} \qquad
   \text{and} \qquad \gamma_{i+1}(G_k)/\gamma_{i+2}(G_k) \cong C_p^{\,
     l(i+1)},
   \]
   where $l(i), l(i+1) \le (2p^k - i)/2$.  Extending the argument one
   step further, we obtain $\gamma_{2p^k}(G_k) = 1$: the group $G_k$
   has nilpotency class at most $2p^k -1$.

   Finally, it suffices to check that the upper bounds that we derived
   from \eqref{equ:cases-I1-I2} and \eqref{equ:cases-I3-I4} for the
   logarithmic orders
   $\log \lvert \gamma_i(G_k) : \gamma_{i+1}(G_k) \rvert$,
   $1 \le i \le 2p^k -1$, sum to the logarithmic order of~$G_k$.
   Indeed, based on Proposition~\ref{pro:order-G-k}, we confirm that
   \[
   (k+1) + \sum_{i=2}^{p^k} \lceil i/2 \rceil +
   \sum_{i=p^k+1}^{2p^k-1} \lceil (2p^k-i)/2 \rceil = k + p^k
   +\binom{p^k}{2} = \log_p \lvert G_k \rvert. \qedhere 
   \]
\end{proof}

\begin{corollary} \label{cor:index} 
  For $i \in \N$ we have
  \[
  \log_p \lvert Z : \gamma_i(G) \cap Z \rvert =
  \begin{cases}
    2 \sum_{j=1}^{(i-3)/2} j = (i^2-4i+3)/4 &
    \text{if $i \equiv_2 1$,} \\
    2 \sum_{j=1}^{(i-4)/2} j + \frac{i-2}{2} = (i^2-4i+4)/4 &
    \text{if $i \equiv_2 0$.}
  \end{cases}
  \]
\end{corollary}

\begin{proof}
  The claim follows from the standard identity
  \[
  \lvert \gamma_2(G) : \gamma_i(G) \rvert = \lvert \gamma_2(G) :
  \gamma_i(G) Z \rvert \lvert \gamma_i(G) Z : \gamma_i(G) \rvert =
  \lvert \gamma_2(W) : \gamma_i(W) \rvert \lvert Z : \gamma_i(G) \cap
  Z \rvert
  \]
  and Propositions~\ref{pro:Wk-lcs} and~\ref{pro:G-k-lcs}.
\end{proof}

From the lower central series of~$G_k$, it is easy to compute the
lower $p$-series and the dimension subgroup series of~$G_k$.

\begin{proposition} \label{pro:Gk-pcs} For $k\in\N$, the $p$-central
  series of $G_k$ has length $2p^k -1$ and its terms satisfy, for $1
  \le i \le 2p^k -1$, 
  \[
  P_i(G_k) = \langle x_k^{\, p^{i-1}} \rangle \gamma_i(G_k).
  \]
\end{proposition}

\begin{proof}
  The description of $P_1(G_k) = \gamma_1(G_k)$ is correct.  Now
  suppose that $i\ge 2$.  By induction, we have
  \[
  P_{i-1}(G_k) = \langle x_k^{\, p^{i-2}} \rangle \gamma_{i-1}(G_k).
  \]
  Recall that $P_i(G_k) = [P_{i-1}(G_k),G_k] P_{i-1}(G_k)^p$ and
  consider the two factors one after the other. The first factor
  satisfies
  \[ 
  [P_{i-1}(G_k),G_k] = [\langle x_k^{\, p^{i-2}} \rangle
  \gamma_{i-1}(G_k),G_k] = [\langle x_k^{\, p^{i-2}} \rangle,G_k]
  \gamma_{i}(G_k),
  \]
  and Lemma~\ref{lem:com-ids} yields
  \[
  [\langle x_k^{\, p^{i-2}} \rangle,G_k] \le [G_k^{\, p^{i-2}},G_k] \le
  \gamma_{p^{i-2}+1}(G_k).
  \]
  From $p^{i-2}+1 \ge i$ we deduce that
  $[P_{i-1}(G_k),G_k]= \gamma_{i}(G_k)$.

  The second factor satisfies
  \[
  P_{i-1}(G_k)^p\equiv \langle
  x_k^{\, p^{i-2}} \rangle^p \, \gamma_{i-1}(G_k)^p \equiv \langle
  x_k^{\, p^{i-1}} \rangle\pmod{\gamma_i(G_k)}.
  \]
  We conclude that
  $P_i(G_k)=\langle x_k^{\, p^{i-1}} \rangle \gamma_i(G_k)$.
\end{proof}

\begin{proposition} \label{pro:Gk-dss} For $k\in\N$, the dimension
  subgroup series of $G_k$ has length $2p^k-1$ and its terms satisfy,
  for $1 \le i \le 2p^k-1$,
  \[
  D_i(G_k) = \langle x_k^{p^{l(i)}}\rangle\gamma_i(G_k), \quad
  \text{where $l(i)=\lceil\log_p(i)\rceil$}.
  \]
\end{proposition}

\begin{proof}
  Let $i \in \N$.  Since $\gamma_2(G_k)$ has exponent~$p$, Lazard's
  formula (see~\cite[Thm.~11.2]{DidSMaSe99}) shows that
  \[
  D_i(G_k) = \prod_{np^m\ge i} \gamma_n(G_k)^{p^m}  =
  G_k^{\, p^{l(i)}} \gamma_i(G_k), \quad \text{where $l(i) = \lceil
    \log_p(i) \rceil$.}
  \]
  Lemma~\ref{lem:com-ids} yields
  $a^{p^{l(i)}} b^{p^{l(i)}} \equiv (ab)^{p^{l(i)}}$
  modulo $\gamma_{p^{l(i)}}(G)$
  for all $a,b \in G_k$ and, as $p^{l(i)} \ge i$, we deduce that
  \[
  D_i(G_k) = \langle x_k^{p^{l(i)}} \rangle \gamma_i(G_k). \qedhere
  \]
\end{proof}


\section{Normal Hausdorff spectra} \label{sec:normal-hspec}

In this section we establish Theorem~\ref{thm:main}; we split the proof into
three parts and formulate three separate results, in dependence on the
filtration series.  We use the notation set up in the introduction; in
particular, $G \cong \varprojlim_k G_k$ denotes the group constructed
there.

\begin{theorem} \label{thm:spec-lcs} The pro-$p$ group $G$ has full
  normal Hausdorff spectra 
  \[
  \hspec_{\trianglelefteq}^{\mathcal{L}}(G)=[0,1] \qquad \text{and}
  \qquad \hspec_{\trianglelefteq}^{\mathcal{D}}(G)=[0,1],
  \]
  with respect to the lower $p$-series~$\mathcal{L}$ and the dimension
  subgroup series~$\mathcal{D}$.
\end{theorem}

\begin{proof}
  Let $\mathcal{S}$ be $\mathcal{L}$, resp.\ $\mathcal{D}$.  Write
  $\mathcal{S} \colon G = S_0 = S_1\ge S_2\ge\ldots$, where
  $S_i = P_i(G)$, resp.\ $S_i = D_i(G)$, for $i \ge 1$, and observe
  that $Z \le \gamma_2(G)$; compare Remark~\ref{rem:H'-eq-Z}.  Thus
  Proposition~\ref{pro:Gk-pcs}, resp.\ Proposition~\ref{pro:Gk-dss},
  yields
  \[
  S_i \cap Z = \gamma_i(G) \cap Z \quad \text{for $i\ge 1$}.
  \]

  From Corollary~\ref{cor:index} we see that
  \begin{equation} \label{equ:lim-zero} \lim_{i\rightarrow\infty}
    \frac{i}{\log_p \lvert Z : \gamma_{i}(G) \cap Z \rvert} = 0.
  \end{equation}
  This allows us to pin down the Hausdorff dimension of
  $Z \le_\mathrm{c} G$:
  \begin{multline*}
    \hdim_{G}^{\mathcal{S}}(Z) = \varliminf_{i \to \infty}
    \left(\frac{\log_p \lvert G : S_i \rvert }{\log_p \lvert S_iZ :
        S_i \rvert }\right)^{-1} = \varliminf_{i \to \infty}
    \left(\frac{\log_p \lvert G : S_iZ \rvert + \log_p \lvert S_iZ :
        S_i\rvert}{\log_p \lvert S_iZ : S_i \rvert}\right)^{-1} \\
    = \varliminf_{i \to \infty}\left( \frac{\log_p \lvert G : S_iZ
        \rvert}{\log_p \lvert Z : S_i \cap Z \rvert} + 1\right)^{-1}
    =\varliminf_{i \to \infty}\left(\dfrac{\log_p \lvert G :
        S_iZ \rvert}{\log_p \lvert Z : \gamma_i(G)\cap Z \rvert} + 1
    \right)^{-1}=1,
  \end{multline*}
  where the last equality follows from~\eqref{equ:lim-zero} and the
  fact that $\log_p \lvert G : S_iZ \rvert \le 2i$, by
  \cite[Prop.~2.6]{KlTh19} and Proposition~\ref{pro:Gk-dss}.  In
  particular, $Z$ has strong Hausdorff dimension.

  Thus Proposition~\ref{pro:path-area}, with $e_i=1$, $n_i=i$ and
  $M_i = \gamma_i(G)$, yields
  \[ 
  [0,1] = [0,\hdim_G^{\mathcal{S}}(Z)] \subseteq
  \hspec_{\trianglelefteq}^{\mathcal{S}}(G). \qedhere
  \]
\end{proof}

\begin{theorem} \label{thm:spec-pp} The pro-$p$ group $G$ has full
  normal Hausdorff spectra
  \[
  \hspec_{\trianglelefteq}^{\mathcal{P}}(G)=[0,1] \qquad \text{and}
  \qquad \hspec_{\trianglelefteq}^{\mathcal{P}^*}(G)=[0,1],
  \]
  with respect to the $p$-power series~$\mathcal{P}$ and the iterated
  $p$-power series $\mathcal{P}^*$.
\end{theorem}

\begin{proof}
  Recall our notation $\pi_i(G) = G^{p^i}$ and $\pi_i^*(G)$ for the
  terms of the series $\mathcal{P}$ and~$\mathcal{P}^*$.  Our first
  aim is to show that
  \begin{equation} \label{equ:rel-p-to-ip} \gamma_{2p^i}(G) \le
    G^{p^i} \le \pi_i^*(G) \le \langle x^{p^i} \rangle \gamma_{p^i}(G)
    \quad \text{for $i \in \N_0$.}
  \end{equation}
 
  Let $i \in \N_0$.  From the construction of $G$ and $G_k$, it is
  easily seen that $G/G^{p^k} \cong G_k/G_k^{p^k}$ for $k \in \N$.
  Hence Proposition~\ref{pro:G-k-lcs} yields
  $\gamma_{2p^i}(G) \le G^{p^i}$.  Clearly, we have
  $G^{p^i} \le \pi_i^*(G)$.  It remains to justify the last inclusion
  in~\eqref{equ:rel-p-to-ip}.  We proceed by induction on~$i$.  For
  $i=0$ even equality holds, trivially.  Now suppose that $i \ge 1$.
  The induction hypothesis yields
  \[
  \pi_{i-1}^*(G) \le \langle x^{p^{i-1}} \rangle \gamma_{p^{i-1}}(G).
  \]
  Let $g \in \pi_{i-1}^*(G)$, and write $g = x^{m p^{i-1}} h$ with
  $m \in \Z$ and $h \in \gamma_{p^{i-1}}(G) \cap H$.
  Lemma~\ref{lem:com-ids} yields $g^p = x^{m p^i} z$ with
  $x^{m p^i} \in \langle x^{p^i} \rangle$ and
  $z \in\gamma_p(\langle x^{p^{i-1}}, h \rangle)$.  Thus it suffices
  to show that
  $\gamma_p(\langle x^{p^{i-1}}, h \rangle) \le \gamma_{p^i}(G)$.
  
  Suppose that $c$ is an arbitrary commutator of weight $n \ge 2$ in
  $\{ x^{p^{i-1}}, h \}$; we show by induction on $n$ that
  $c \in \gamma_{np^{i-1}}(\langle x^{p^{i-1}}, h \rangle)$.  For
  $n=2$, it suffices to consider $c = [h,x^{p^{i-1}}]$, and
  Lemma~\ref{lem:com-ids} shows that $c \in \gamma_{2p^{i-1}}(G)$.
  For $n\ge 3$, we see by induction that it suffices to consider
  $c = [d,h]$ and $[d, x^{p^{i-1}}]$ with
  $d \in \gamma_{(n-1)p^{i-1}}(G)$; if $c = [d,h]$, the result follows
  immediately, and, if $c = [d,x^{p^{i-1}}]$, the result follows again
  by Lemma~\ref{lem:com-ids}.  This concludes the proof
  of~\eqref{equ:rel-p-to-ip}.

  Let $\mathcal{S} = \mathcal{P}$, resp.\
  $\mathcal{S} = \mathcal{P}^*$, and write $S_i = \pi_i(G) = G^{p^i}$,
  resp.\ $S_i = \pi_i^*(G)$, for $i \in \N_0$.  Recall that
  $Z \le \gamma_2(G)$; compare Remark~\ref{rem:H'-eq-Z}.  Thus
  \eqref{equ:rel-p-to-ip} yields
  \begin{equation} \label{equ:Si-cap-Z-estimate} \gamma_{2p^i}(G) \cap
    Z \le S_i \cap Z \le \big(\langle x^{p^i} \rangle \gamma_{p^i}(G)
    \big) \cap Z =\gamma_{p^i}(G) \cap Z.
  \end{equation}
  From Corollary~\ref{cor:index} we see that
  \begin{equation} \label{equ:lim-zero-2}
    \lim_{i \to \infty}\frac{2p^i}{\log_p \lvert
      Z:\gamma_{p^i}(G) \cap Z \rvert}=0.
  \end{equation}
  As in the proof of Theorem~\ref{thm:spec-lcs} we want to apply
  Proposition~\ref{pro:path-area}, here with $e_i = 1$,
  $n_i = 2p^i$ and $M_i = \gamma_{p^i}(G)$, to conclude that $G$ has
  full normal Hausdorff spectrum.

  It remains to check that $\hdim_G^{\mathcal{S}}(Z)=1$.  We observe
  that, for $i \in \N_0$,
  \[
  \log_p \lvert G : S_i Z \rvert \le \log_p \lvert G_i : G_i^{\, p^i}
  Z_i \rvert \le \log_p \lvert W_i \rvert = i + p^i \le 2p^i,
  \]
  and thus, by~\eqref{equ:Si-cap-Z-estimate} and~\eqref{equ:lim-zero-2},
  \[
  \lim_{i \to \infty} \frac{\log_p \lvert G : S_i Z
    \rvert}{\log_p \lvert Z : S_i\cap Z \rvert} \le
  \lim_{i \to \infty} \frac{\log_p \lvert G : S_i Z
    \rvert}{\log_p \lvert Z : \gamma_{p^i}(G) \cap Z \rvert}  = 0.
  \]
  As in the proof of Theorem~\ref{thm:spec-lcs} we conclude that
  $\hdim_{G}^{\mathcal{S}}(Z) = 1$.
\end{proof}

A little extra work is required to determine the normal Hausdorff
spectrum of~$G$ with respect to the Frattini series.   We define
\[
z_{i,j} =
\begin{cases}
  [c_i,c_j] \in\gamma_{i+j}(G) & \text{for $i,j \ge 1$,} \\
  1 & \text{otherwise.}
\end{cases}
\]
Proposition~\ref{pro:Wk-lcs} and Remark~\ref{rem:derived-subgroup}
show that
\[
H =\langle c_i \mid i \ge 1 \rangle \qquad \text{and} \qquad
Z = \langle z_{i,j} \mid 1\le j<i \rangle.
\]
Moreover, from Corollary~\ref{cor:index} it can be seen that, for $k \ge 2$,
\begin{equation} \label{equ:gamma-cap-Z} \gamma_k(G) \cap Z = \langle
  z_{i,j} \mid 1 \le j <i \text{ and } i+j \ge k \rangle.
\end{equation}

\begin{lemma} \label{lem:double-prod} For $i,j\in\N$ and $r \in \N_0$,
  the following identity holds:
  \[
  [z_{i,j}, x, \overset{r}{\ldots}, x] = \prod_{s=0}^r 
    \prod_{t=0}^s \, z_{i+r-t,j+r-s+t}^{\, \binom{r}{s} \binom{s}{t}}.
  \]
\end{lemma}

\begin{proof}
  We argue by induction on~$r$.  For $r=0$ both sides are equal
  to~$z_{i,j}$.  Now suppose that~$r \ge 1$.  We observe that, for
  $m,n \ge 1$,
  \begin{equation} \label{equ:z-rel} [z_{m,n},x] = z_{m,n}^{\, -1}
    [c_m^{\, x}, c_n^{\, x}] = z_{m,n}^{\, -1} [c_m c_{m+1}, c_n
    c_{n+1}] = z_{m+1,n} \, z_{m,n+1} \, z_{m+1,n+1}.
  \end{equation}
  Thus the induction hypothesis yields
  \[ 
  [z_{i,j}, x, \overset{r}{\ldots}, x] = [[z_{i,j}, x,
  \overset{r-1}{\ldots}, x],x] = \prod_{s=0}^r \prod_{t=0}^{s} \,
  [z_{i+r-1-t,j+r-1-s+t},x]^{\binom{r-1}{s} \binom{s}{t}} ,
  \]
  and, in view of~\eqref{equ:z-rel}, the result follows from the
  identity
  \begin{multline*}
    \binom{r-1}{s-1}\binom{s-1}{t} + \binom{r-1}{s-1}\binom{s-1}{t-1} +
    \binom{r-1}{s}\binom{s}{t}\\
    =\binom{r-1}{s-1}\binom{s}{t}+
    \binom{r-1}{s}\binom{s}{t} = \binom{r}{s} \binom{s}{t}
  \end{multline*}
  for $0 \le s \le r$ and $0 \le t \le s$.
\end{proof}

Lemma~\ref{lem:com-ids} and Lemma~\ref{lem:double-prod} lead directly to a
useful corollary.

\begin{corollary} \label{cor:p-k}
 For  $i,j\in\N$ and $k \in \N_0$, the following identity holds:
 \[ 
 [z_{i,j},x^{p^k}]=z_{i+{p^k},j}z_{i,j+p^k}z_{i+p^k,j+p^k}.
 \]
\end{corollary}

\begin{theorem} \label{thm:spec-frattini} The pro-$p$ group $G$ has
  full normal Hausdorff spectrum
  \[
  \hspec_{\trianglelefteq}^{\mathcal{F}}(G)=[0,1],
  \]
  with respect to the Frattini series~$\mathcal{F}$.
\end{theorem}

\begin{proof}
  For $i \in \N_0$, we write $[i]_p = (p^i-1)/(p-1)$ and note, for
  $i \ge 1$, that $[i-1]_p + p^{i-1} = [i]_p$.  We consider
  \[
  C_i = \langle x^{p^i} \rangle \ltimes \langle c_j \mid j \ge
  1 + [i]_p \rangle \le_\mathrm{c} G
  \]
  and claim, for $i \ge 1$, that
  \begin{equation} \label{equ:Ck-claim} \Psi_i^-(G) \le \Phi_i(G)
    \le \Psi_i^+(G),
  \end{equation}
  where
  \[
  \Psi_i^-(G) = C_i \big(
  \gamma_{1+2 [i-1]_p +p^{i-1}}(G) \cap Z \big) \quad
  \text{and} \quad \Psi_i^+(G) = C_i \big(
  \gamma_{2 + 2 [i-1]_p}(G) \cap Z \big).
  \]
  For $i=1$ the assertion is that
  $\Phi(G) = C_1 (\gamma_2(G) \cap Z) = \langle x^p, c_2, c_3 , \ldots
  \rangle (\gamma_2(G) \cap Z)$,
  which follows from Proposition~\ref{pro:Wk-lcs} and the fact that
  $Z \le \gamma_2(G)$.  Now suppose that $i \ge 2$.
  Lemma~\ref{lem:com-ids} and the observation that
  $p^{i-1} \ge 2 p^{i-2}$ yield
  \[
  [\gamma_{2+2 [i-2]_p}(G) \cap Z,x^{p^{i-1}}] \le \gamma_{2+2 [i-2]_p
    +p^{i-1}}(G) \cap Z \le \gamma_{2+2 [i-1]_p}(G) \cap Z;
  \]
  by construction, we have
  $[\gamma_{2 + 2 [i-2]_p}(G) \cap Z, c_n] = 1$ for all $n \ge 1$.
  Furthermore, Lemma~\ref{lem:com-ids} gives
  \begin{equation} \label{equ:c-n-x-comm} [c_n,x^{p^{i-1}}] \equiv
    c_{n+{p^{i-1}}} \pmod{\gamma_{2n+p^{i-1}}(G)\cap Z} \quad
    \text{for all $n \ge 1$,}
  \end{equation}
  and hence
  \[ 
  [C_{i-1},x^{p^{i-1}}] \le C_i\big(\gamma_{2 + 2 [i-1]_p +
    p^{i-1}}(G) \cap Z \big).
  \]
  By induction,
  $\Phi_{i-1}(G) \le \Psi_{i-1}^+(G) = C_{i-1} \big(\gamma_{2+2
    [i-2]_p}(G) \cap Z \big)$, and this implies
  \begin{multline*}
    \Phi_i(G) = \Phi(\Phi_{i-1}(G)) \le \langle x^{p^i} \rangle
    [C_{i-1},C_{i-1}] \big(\gamma_{2 + 2 [i-1]_p}(G) \cap Z \big) \\ \le
    C_i\big(\gamma_{2 + 2 [i-1]_p}(G) \cap Z \big) = \Psi_i^+(G).
  \end{multline*}
  
  It remains to check the first inclusion in~\eqref{equ:Ck-claim}; by
  induction, it suffices to show that
  \[
  \Psi_i^-(G) \le K, \quad \text{where
    $K = \Phi\big( \Psi_{i-1}^-(G) \big).$}
  \]
  First we show that $\gamma_{1+2 [i-1]_p + p^{i-1}}(G)\cap Z \le K$
  implies $C_i \le K$.  Clearly, $x^{p^i} \in C_{i-1}^{\, p} \le K$,
  and \eqref{equ:c-n-x-comm} shows that, for $j \ge 1 + [i]_p$, there
  exists
  $d_j \in \gamma_{2 (j-p^{i-1}) + p^{i-1}}(G) \cap Z \le \gamma_{1 +
    2 [i-1]_p + p^{i-1}}(G) \cap Z$ such that
  \[
  c_j = [c_{j-p^{i-1}}, x^{p^{i-1}}] d_j \in [C_{i-1},C_{i-1}] \le K.
  \]
  Thus it suffices to prove that
  $\gamma_{1+2 [i-1]_p + p^{i-1}}(G) \cap Z \le K$.

  From \eqref{equ:gamma-cap-Z} we recall that
  \[
  \gamma_{1+ 2 [i-1]_p + p^{i-1}}(G) \cap Z = \langle z_{j,k} \mid 1
  \le k < j \text{ and } j+k \ge 1+2 [i-1]_p +p^{i-1} \rangle.
  \]
  From $[C_{i-1},C_{i-1}] \le K$ we deduce that 
  \begin{equation} \label{equ:z-m-n-criterion} z_{m,n} \in K \quad
    \text{for $m > n \ge 1 + [i-1]_p$.}
  \end{equation}
  Thus, it remains to see that $z_{j,k} \in K$ for $j, k \in \N$
  satisfying
  \[
  1 \le k < j, \qquad j+k \ge 1 + 2 [i-1]_p + p^{i-1} \qquad \text{and}
  \qquad k \le [i-1]_p.
  \]
  Given such $j, k \in \N$, we observe that
  \[
  k < 1 + [i-1]_p \le j - p^{i-1} \qquad \text{and} \qquad (j - p^{i-1})
  + k \ge 1 + 2 [i-1]_p;
  \]
  hence \eqref{equ:gamma-cap-Z} implies
  \[
  z_{j - p^{i-1},k} \in \gamma_{1 + 2 [i-1]_p}(G) \cap Z \le \gamma_{1+2
    [i-2]_p +p^{i-2}}(G) \cap Z \le \Psi_{i-1}^-(G).
  \]
  We apply Corollary~\ref{cor:p-k} to deduce that
  \begin{equation} \label{equ:z-product} z_{j,k} \,
    z_{j - p^{i-1}, k + p^{i-1}} \, z_{j, k + p^{i-1}} =
    [z_{j - p^{i-1},k}, x^{p^{i-1}}] \in [\Psi_{i-1}^-(G), C_{i-1}] \le
    K.
  \end{equation}

  As $j > k + p^{i-1} \ge 1 + [i-1]_p$, we see from
  \eqref{equ:z-m-n-criterion}, for $m=j$ and $n=k + p^{i-1}$ that
  $z_{j, k + p^{i-1}} \in K$.  Similarly, we deduce that
  $z_{j - p^{i-1}, k + p^{i-1}} \in K$, if $j - p^{i-1} > k+p^{i-1}$, and,
  finally,
  $z_{j - p^{i-1}, k + p^{i-1}} = z_{k + p^{i-1}, j - p^{i-1}}^{\, -1} \in K$,
  if $j-p^{i-1} \le k + p^{i-1}$ and thus $j - p^{i-1} \ge 1+ [i-1]_p$.
  Feeding this information into \eqref{equ:z-product}, we obtain
  $z_{j,k} \in K$ which concludes the proof of~\eqref{equ:Ck-claim}.

  From~\eqref{equ:Ck-claim} we deduce that
  \[
  \gamma_{1 + 2 [i-1]_p + p^{i-1}}(G) \cap Z \le \Phi_i(G)
  \cap Z \le \gamma_{2 + 2 [i-1]_p}(G) \cap Z,
  \]
  and from Corollary~\ref{cor:index} we see that
  \[
  \lim_{i \to \infty} \frac{2 [i-1]_p + p^{i-1}}{\log_p \vert Z :
    \gamma_{2+2 [i-1]_p}(G) \cap Z \rvert} = 0.
  \]
  As in the proof of Theorem~\ref{thm:spec-lcs} we want to apply
  Proposition~\ref{pro:path-area}, here with $e_i = 1$,
  $n_i = 2 [i-1]_p +p^{i-1}$ and $M_i = \gamma_{2 + 2 [i-1]_p}(G)$, to
  conclude that $G$ has full normal Hausdorff spectrum.

  It remains to check that $\hdim_G^{\mathcal{F}}(Z)=1$.  From
  \cite[Prop.~2.6]{KlTh19} we see that
  $\log_p \lvert G : \Phi_i(G) Z \rvert = i + [i]_p$, and hence
  Corollary~\ref{cor:index} implies
  \[
  \lim_{i \to \infty} \frac{\log_p \lvert G : \Phi_i(G)Z
    \rvert}{\log_p \lvert Z : \Phi_i(G) \cap Z \rvert} = 0.
  \]
  As in the proof of Theorem~\ref{thm:spec-lcs} we see that
  $\hdim_G^{\mathcal{F}}(Z)=1$.
\end{proof}

Theorem~\ref{thm:main} summarises the results in
Theorems~\ref{thm:spec-lcs}, \ref{thm:spec-pp}
and~\ref{thm:spec-frattini}.


\end{document}